\documentclass[10pt,a4paper]{article}

\usepackage[english]{babel}
\usepackage[T1]{fontenc}

\usepackage{xcolor}
\usepackage[font=small,labelfont=sc]{subcaption}
\usepackage{amsmath}
\usepackage{amssymb}
\usepackage{amsthm}
  \theoremstyle{plain}
  \newtheorem{theorem}{Theorem}
  \newtheorem{lemma}{Lemma}

  \newtheorem{definition}{Definition}[section]
  
  \theoremstyle{remark}
  \newtheorem*{remark}{Remark}
\usepackage{mathrsfs}
\usepackage{mathtools}
\usepackage[left=2.7cm,right=2.7cm,top=3.0cm,bottom=3.0cm]{geometry}
\usepackage[round]{natbib}
\usepackage[colorlinks=true, allcolors=blue, hypertexnames=false]{hyperref}
\usepackage{booktabs}
\usepackage{algorithmic}
\usepackage{algorithm}
\usepackage{placeins}

\allowdisplaybreaks[4]

\begin{document}

\title{\bf Dynamic mean--variance portfolio selection with no-shorting constraints and unknown investment opportunity sets}

\author{Xun Li\thanks{%
Department of Applied Mathematics,
The Hong Kong Polytechnic University,
Kowloon, Hong Kong, China.
E-mail: li.xun@polyu.edu.hk}
\and
Yutian Wang\thanks{%
Department of Applied Mathematics,
The Hong Kong Polytechnic University,
Kowloon, Hong Kong, China.
E-mail: yutian.wang@connect.polyu.hk}
\and
Xun Yu Zhou\thanks{%
Department of Industrial Engineering and Operations Research \& Data Science Institute,
Columbia University, New York, USA, NY 10027.
E-mail: xz2574@columbia.edu}}

\date{\today}

\maketitle

\begin{abstract}
We study continuous-time mean--variance portfolio selection with no-shorting constraints and unknown investment opportunity sets from a reinforcement learning (RL) perspective.
The problem is a {\it constrained} stochastic linear--quadratic control problem for which the entropy-regularized exploratory  formulation of \citet{wangReinforcementLearningContinuous2020} leads to difficulty in theoretical analysis, because enforcing the constraint on the support of randomized policies nullifies the tractable Gaussian exploration.
To tackle this challenge,
we introduce an auxiliary exploratory problem {\it without} entropy in which exploratory policies are still Gaussian whose samples may violate the no-shorting requirement
but their means satisfy it.
We then prove that,
for a suitable choice of exploration variance,
the mean of the optimal Gaussian policy of the auxiliary problem coincides with the optimal policy of the original problem.
Motivated by this theoretical result,
we develop a model-free RL algorithm that learns the optimal policy of the auxiliary (and hence the original) problem directly from trajectory data without estimating the investment opportunity set.
A numerical example demonstrates  the performance of the proposed algorithm.
\\[1em]
\end{abstract}



\section{Introduction}
Mean--variance (MV) portfolio selection
introduced by \citet{markowitzPortfolioSelection1952} is a cornerstone of the modern quantitative finance theory and practice. The original model of \citet{markowitzPortfolioSelection1952} is static, and there has been a sizable literature
developed over the last 26 years on the continuous-time extension of the MV theory, starting from \citet{zhouContinuoustimeMeanvariancePortfolio2000}.
The main idea of \citet{zhouContinuoustimeMeanvariancePortfolio2000} is to
turn continuous-time  MV into an {\it indefinite} stochastic linear--quadratic (LQ) control problem, whose theory was developed in the 1990s (\citealt{chen1998stochastic}), leading to explicit characterization of the efficient frontier. One of the focuses of the subsequent study of continuous-time MV is to incorporate various portfolio constraints such as those of no-shorting and no-bankruptcy. 
Those constraints translate into control and state constraints, which introduce great technical difficulties because the general LQ control theory does not cope with any constraints.
New and ad hoc approaches have been developed exploiting the special features of the MV problems, including that the state (wealth) is scalar-valued.
%
%
\citet{liDynamicMeanvariancePortfolio2002} solve continuous-time MV portfolio selection when shorting is disallowed, by introducing  two Riccati equations and a viscosity solution analysis of the associated HJB equation.
\citet{huConstrainedStochasticLQ2005} further develop a general theory on cone-constrained stochastic LQ control with random coefficients,
obtain explicit optimal controls via two extended stochastic Riccati equations and apply the results to MV under no-shorting constraints.
\citet{liContinuoustimeMarkowitzsModel2016} investigate  a continuous-time Markowitz model with bankruptcy prohibition and convex cone portfolio constraints,
and derive semi-analytical solutions via a martingale approach without relying on viscosity solutions.

The study on continuous-time MV prior to the year 2020 had been entirely model-based, namely, under the assumption that the investment opportunity set (consisting of the asset dynamics coefficients) is given and accurate. However, it is well understood that it is hard, if not impossible, to accurately
estimate the moments of asset returns, especially the first moment, based on price data which are typically very limited.
On the other hand, portfolios and their objective values derived from analytical/numerical solutions of the MV problems are known to be extremely sensitive to those estimation errors.
These render the derived model-based portfolio policies often irrelevant or even misleading.

This is where reinforcement learning (RL) comes to rescue. RL studies how an agent learns to make sequential decisions from interaction with an {\it unknown} environment.
Primarily premised upon dynamic programming through value-based and temporal-difference methods,
the RL theory and algorithms have been developed mainly for problems in discrete time with discrete state and action (control) spaces;
see, for example,
the classical textbook \citet{suttonReinforcementLearningIntroduction2018}.
\citet{wangReinforcementLearningContinuous2020} are the first to formulate the continuous-time RL with possibly continuous state/action spaces, using relaxed controls and entropy regularization to model exploration when the model parameters are unknown. In particular, they prove theoretically that
in the LQ setting
the optimal exploratory feedback policy must be Gaussian. This provides the interpretability of the Gaussian exploration commonly employed in practice, which is easy to sample from, differentiate and optimize.
\citet{wangContinuoustimeMeanVariance2020} immediately apply this general theory to the continuous-time MV portfolio selection problem with an unknown investment opportunity set to
prove that the optimal exploratory policy is Gaussian with a time-decaying variance,
and establish a policy improvement theorem leading to an implementable RL algorithm.
Subsequently,
\cite{jiaPolicyEvaluationTemporaldifference2022,jiaPolicyGradientActorcritic2022,jiaQlearningContinuousTime2023} develop a systematic theory  on policy evaluation and temporal-difference learning,
policy gradient and actor--critic methods,
and continuous-time $q$-learning,
respectively.
\citet{tangExploratoryHJBEquations2022} establish the well-posedness,
regularity,
and vanishing-exploration convergence for the exploratory HJB equations.
More recently,
\citet{huangConvergencePolicyIteration2025} investigate the convergence of policy iteration for general entropy-regularized stochastic control problems, and \citet{sethiEntropyAnnealingPolicy2025} study
related algorithmic and asymptotic properties of
continuous-time policy mirror descent with entropy annealing.
\citet{guoFastPolicyLearning2026} propose a regularized policy gradient 
method for infinite-horizon LQ control and establish linear
and partly superlinear  convergence. For the continuous-time MV problem, \citet{huangmv2024}
develop a model-free, data-driven RL algorithm based on the policy gradient theory of \citet{jiaPolicyGradientActorcritic2022}, prove not only its convergence but also a {\it sublinear} regret bound in terms of the Sharpe ratio, and carry out an extensive empirical study to demonstrate its superior performance.


The relaxed or exploratory controls are in general probability distributions on action/control spaces. If action spaces are unconstrained, then the resulting exploratory problems are more likely to have better analytical structures leading to more efficient learning algorithms. For example, the optimality of Gaussian exploration
proved in \cite{wangReinforcementLearningContinuous2020}, \cite{wangContinuoustimeMeanVariance2020} and \cite{huangmv2024}
stems not only from the underlying LQ framework but also from the unconstrained action spaces.
When there are constraints the exploratory controls will be supported on the constrained action space. This likely leads to {\it truncated} Gaussian-type exploratory policies for the current problem of MV portfolio selection without shorting. Indeed, in a different setting of the Merton problem, \citet{chauContinuoustimeOptimalInvestment2026} show that the optimal exploratory policy is truncated Gaussian under interval constraints.
Such a formulation ensures that {\it every} control sampled from the exploratory policy is feasible, which however comes with a major drawback:
The HJB equation and the corresponding learning problem become substantially harder to analyze,
and the explicit structure available under Gaussian exploration is generally lost.

To overcome this difficulty,
we approach the constrained problem by imposing the constraint on the policy {\it mean} rather than on the support.
More precisely,
we keep any exploratory policy Gaussian,
but require its mean to lie in the positive half-space.
Hence the randomized policy may place some mass outside $\mathbb{R}_+^m$,
while its ``averaged''  policy is componentwise nonnegative and therefore admissible for the original problem.
We thus formulate an auxiliary exploratory problem where the class of admissible policies are those Gaussian distributions with nonnegative means and a suitable choice of variance.
To justify the relevance of this auxiliary problem vis-\`a-vis the original problem,
our main theoretical result establishes  that the mean taken from
the optimal Gaussian policy of the auxiliary problem is exactly the optimal control of the original no-shorting MV problem.
In other words,
although individual controls  sampled from an exploratory policy may not be feasible, they are for the training purpose only. The policy that is ultimately extracted and implemented from that distribution is not only admissible, but indeed is optimal.
This underpins theoretically  the connection between tractable Gaussian exploration and the original constrained problem.

It should be noted that the above idea was partially inspired by \cite{daiDatadrivenMertonsStrategies2025} where a Merton problem is studied. The problem therein does not generate Gaussian exploration if one follows the entropy-regularized formulation of \cite{wangReinforcementLearningContinuous2020}. Instead, the authors propose a different exploratory formulation {\it without} entropy, which is to randomize controls by simply adding proper zero-mean Gaussian noises to them. They then prove that the mean of their optimal exploratory policy solves the original Merton problem. The problem in \cite{daiDatadrivenMertonsStrategies2025} is unconstrained, but here we prove that the key idea can be extended to our constrained problem.

Based on the established theoretical result, we only need to solve the auxiliary exploratory problem for which existing RL theory and algorithms can be adapted to.
Indeed, we devise a model-free RL learning algorithm following martingale-based policy evaluation from \citet{jiaPolicyEvaluationTemporaldifference2022} and
policy gradient  from \citet{jiaPolicyGradientActorcritic2022}.
The algorithm  learns the optimal no-shorting portfolio policies from sample paths (i.e. data) without  estimation of the investment opportunity set. We also present a numerical example that demonstrates that the algorithm converges and learns the optimal oracle values effectively.



The rest of the paper is organized as follows.
In Section~2,
we formulate the target mean--variance problem with no-shorting constraints and unknown model parameters and introduce the auxiliary exploratory problem.
In Section~3,
we carry out a theoretical analysis on the auxiliary problem and establish its relation to the target problem.
In Section~4,
we develop a model-free RL  algorithm for solving the auxiliary (and hence the target) problem.
Section~5 presents a numerical example. Section 6 finally concludes.

\section{Problem Formulation}
In this section, we state the mean--variance problem with no-shorting constraints and unknown parameters (the \emph{target problem}),
and then formulate an associated exploratory problem with integral constraints that can be solved in a data-driven manner by an RL algorithm developed in this paper (the \emph{auxiliary problem}).
The formulation of the auxiliary problem, while originally rooted in the entropy regularization
framework of \citet{wangContinuoustimeMeanVariance2020}, is directly inspired by \citet{daiDatadrivenMertonsStrategies2025} that employs
a different way
to carry out exploration {without} involving entropy. This choice of formulation is key to treating the no-shorting constraint.

\subsection{Notation}
Frequently used notation in this paper includes:
\begin{itemize}
\item $M^\intercal$: the transpose of any matrix or vector $M$;
\item $\|M\|$: $\sqrt{\sum_{i,j} m_{ij}^2}$ for any matrix or vector $M=(m_{ij})$;
\item $\mathbb{R}^n$: \(n\)-dimensional real Euclidean space;
\item $\mathbb{R}_+^n$: the subset of  $\mathbb{R}^n$ consisting of elements with nonnegative components;
\item $\mathcal{N}(\mu, \Sigma)$: the Gaussian density function with mean $\mu$ and covariance $\Sigma$;
\item $\mathcal{N}(\mu, 0)$: the Dirac distribution centered at $\mu$;
\item $\mathcal{P}(\mathbb{R}^n)$: the set of density functions on $\mathbb{R}^n$;
\item $\mathcal{P}_+(\mathbb{R}^n)$: the subset of $\mathcal{P}(\mathbb{R}^n)$ consisting of distributions with nonnegative mean;
\item $\mathbb{S}^n_{++}$: the set of strictly positive definite matrices in $\mathbb{R}^{n \times n}$;
\item $a \geq 0$: a componentwise inequality if $a$ is a vector.
\end{itemize}

Let $T > 0$ be the investment horizon, and $\{W_t,~0\leq t \leq T\}$ be a standard
\(m\)-dimensional Brownian motion defined on a filtered probability space $(\Omega, \mathcal{F}, \mathcal{F}_{0\leq t\leq T}, \mathbb{P})$
that satisfies the usual conditions. Given a Hilbert space $\mathcal{H}$ with the norm $\|\cdot\|_\mathcal{H}$,
define the Banach space
\begin{equation*}
\begin{aligned}
\mathcal{L}^2_\mathcal{F}(0,T;\mathcal{H}) &:= \biggl\{\varphi(\cdot) \biggm| \varphi \text{ is an
\(\{\mathcal{F}_t\}_{0\leq t\leq T}\)-adapted, \(\mathcal{H}\)-valued measurable process}\\
&\hspace{6em} \text{ and }
\mathbb{E}\int_0^T\|\varphi(t)\|_\mathcal{H}^2\,dt < \infty\biggr\}
\end{aligned}
\end{equation*}
with the norm $\|\varphi\| := \sqrt{ \mathbb{E}\int_0^T \|\varphi(t)\|_\mathcal{H}^2\,dt }$.

\subsection{Mean--Variance Portfolio Selection with No-Shorting Constraints}
We consider investing $m$ risky assets whose prices are geometric Brownian motion governed by
\begin{equation*}
dS^{i}_t = S^{i}_t b^{i}_t \,dt + S^{i}_t \sum_{j=1}^m \sigma^{ij}_t\,dW^{j}_t, \quad t\in[0,T],
~S_0^{i} \in \mathbb{R},
\quad i\in\{1,2,\ldots,m\},
\end{equation*}
and one riskless asset whose price follows
\begin{equation*}
dS^{0}_t = S^{0}_t r_t\,dt, \quad t\in[0,T],~S_0^0 \in \mathbb{R}.
\end{equation*}
Denote the appreciation rate vector by $b_t := (b^{1}_t,b^{2}_t,\ldots,b^{m}_t)^\intercal$ and
the volatility matrix by $\sigma_t := (\sigma_t^{ij})$.
Throughout this paper,
we impose the following standing assumptions:
1) the coefficients $r_t$, $b_t$, and $\sigma_t$ are deterministic, measurable, and bounded functions of $t$ on $[0,T]$; and
2) there exists a constant $\delta>0$ such that $\sigma_t \sigma_t^\intercal - \delta I$ is strictly positive definite for all $t\in[0,T]$.

Let $u(\cdot)=\{u_t,~0\leq t \leq T\}$ be the portfolio (process), where
$u_t:=(u^{1}_t,u^{2}_t,\ldots,u^{m}_t)^\intercal$ with
$u^{i}_t$ being the \emph{discounted dollar value}  in the \(i\)-th risky asset at $t$.
This paper studies the case when shorting is disallowed, i.e., $u_t^{i} \geq 0,~i\in\{1,2,\ldots,m\},~t\in [0,T]$. The set of admissible portfolios is defined to be
$\mathcal{L}_{\mathcal{F}}^2(0,T; \mathbb{R}_+^m)$.
For any $u(\cdot) \in \mathcal{L}_{\mathcal{F}}^2(0,T; \mathbb{R}_+^m)$, the \emph{discounted wealth process} $x(\cdot)=\{x_t,~0\leq t \leq T\}$ satisfies
\begin{equation}
\label{eq:wealth-equation}
dx_t = u_t^\intercal \sigma_t (\rho_t\,dt + dW_t),\quad t\in[0,T],~x_0 \in \mathbb{R},
\end{equation}
where $\rho_t := \sigma_t^{-1}(b_t - r_t \mathbf{1})$ is the market price of risk at $t$.

The mean--variance model under no-shorting constraints is formulated as
\begin{equation}
\label{eq:mv-problem}
\min_{u(\cdot) \in \mathcal{L}_{\mathcal{F}}^2(0,T; \mathbb{R}_+^m)} \operatorname{Var}(x_T),
\quad \mbox{subject to (s.t.) } \quad \mathbb{E}[x_T] = z,
\end{equation}
where the wealth--portfolio pair $(x(\cdot),u(\cdot))$ satisfies \eqref{eq:wealth-equation},
and $z \in \mathbb{R}$ is a pre-specified (discounted) target expected return.

Following
\citet{zhouContinuoustimeMeanvariancePortfolio2000}, we introduce
a Lagrange multiplier $w \in \mathbb{R}$ for the equality constraint $\mathbb{E}[x_T] = z$.
The problem \eqref{eq:mv-problem} is thus transformed into a
standard stochastic optimal control problem
with a parameter $w \in \mathbb{R}$:
\begin{equation}
\label{eq:problem-w-z}
\min_{u(\cdot) \in \mathcal{L}_{\mathcal{F}}^2(0,T; \mathbb{R}_+^m)} J\bigl(u(\cdot);w\bigr):= \mathbb{E}[(x_T-w)^2] - (w-z)^2,
\quad \operatorname{s.t.} \quad \eqref{eq:wealth-equation}.
\end{equation}
For each fixed $w\in \mathbb{R}$,
let $(x^*(\cdot;w),u^*(\cdot;w))$ denote the optimal wealth--portfolio process of \eqref{eq:problem-w-z}.
Then
the correct multiplier $w^*$ can be determined by
\begin{equation}
\label{eq:def-w-star}
\mathbb{E}[x^*(T;w^*)] = z.
\end{equation}
In particular,
when $w=w^*$,
the problems \eqref{eq:mv-problem} and \eqref{eq:problem-w-z} have the same optimal portfolio  and the same optimal objective value.

The above formulation of the problem is the same as that in \citet{liDynamicMeanvariancePortfolio2002}, except now we do not know the coefficient functions (or the investment opportunity set)
$(b_t,r_t,\sigma_t)$.

\subsection{An Auxiliary Problem with Randomized Portfolios}
The key feature of this paper is that the model parameters of the problem \eqref{eq:mv-problem} are unknown. To cope with it we employ the idea of exploratory (randomized) formulation originally proposed by
\citet{wangContinuoustimeMeanVariance2020} for continuous-time RL. However, directly using the entropy-regularization technique therein runs into the problem of truncated distribution for portfolio randomization due to the no-shorting constraint, as discussed in Introduction.
Instead, we follow \citet{daiDatadrivenMertonsStrategies2025} to propose
an auxiliary problem which is solvable by RL algorithms and, more importantly, whose solution leads directly to the solution to
the target problem \eqref{eq:mv-problem}.

In the auxiliary problem, the portfolio $u(\cdot)$ is replaced by a randomized portfolio process  $\pi(\cdot)$,
which is a distribution (density-function)-valued, progressively measurable stochastic process
$\pi(\cdot):=\{\pi_t,~0\leq t \leq T\}$.
Introduce the nonnegative integral constraints on $\pi(\cdot)$:
\begin{equation*}
\int_{\mathbb{R}^m} u\,\pi_t(u)\,du \geq 0,\qquad \forall t\in[0,T].
\end{equation*}
In other words,
we require that at each time instant $t$, the distribution $\pi_t$ must have a nonnegative {\it mean}. Any such portfolio is called {\it admissible}. Note that an individual portfolio sampled from an admissible $\pi(\cdot)$ may {\it violate} the no-shorting constraint.

Recall that  $\mathcal{P}_+(\mathbb{R}^m) = \{\pi \in \mathcal{P}(\mathbb{R}^m) \mid \int_{\mathbb{R}^m} u\,\pi(u)\,du \geq 0 \}$.
Then,
the above constraint can also be written as $\pi_t \in \mathcal{P}_+(\mathbb{R}^m)\;\;\forall t\in[0,T]$. Clearly,
$\mathcal{P}(\mathbb{R}^m_+) \subset \mathcal{P}_+(\mathbb{R}^m) \subset \mathcal{P}(\mathbb{R}^m)$.

Given a randomized portfolio $\pi(\cdot)$, introduce the exploratory dynamics following  \cite{wangReinforcementLearningContinuous2020}:
\begin{equation}
\label{eq:exploratory-wealth-equation}
dx^\pi_t = \tilde{b}_t(\pi_t)\,dt + \sum_{j=1}^m\tilde{\sigma}^j_t(\pi_t)
\,dW_t^j,\quad t\in[0,T],
\end{equation}
where $\tilde{b}_t(\cdot)$ and $\tilde{\sigma}_t(\cdot)$ are functions on $\mathcal{P}(\mathbb{R}^m)$ defined by
\begin{equation*}
\tilde{b}_t(\pi) := \int_{\mathbb{R}^m} (\sigma_t\rho_t)^\intercal u\,\pi(u)\,du,\qquad
\tilde{\sigma}^j_t(\pi) := \biggl(\int_{\mathbb{R}^m}
\Bigl(\sum_{i=1}^m \sigma_t^{ij}u_i\Bigr)^2
\,\pi(u)\,du \biggr)^{1/2}.
\end{equation*}

We now specify the set of randomized feedback portfolios ({\it policies}) for our auxiliary problem.
A feedback policy $\pi: [0,T] \times \mathbb{R} \to \mathcal{P}_+(\mathbb{R}^m)$ is called {\it admissible}
if the generated process $\{\pi(t,x_t^\pi),~0\leq t \leq T\}$ is admissible (in particular it satisfies the nonnegative mean constraint), where $\{x_t^\pi,~0\leq t \leq T\}$ is the solution to \eqref{eq:exploratory-wealth-equation} with
$\pi_t$ therein substituted with $\pi(t,x_t^\pi)$.
Fix a measurable
function $\Sigma: [0,T] \times \mathbb{R} \to \mathbb{S}^{m}_{++}$, and denote by
$\Pi$ the set of admissible \emph{Gaussian} policies $\pi$ where $\pi(t,x) = \mathcal{N}(a(t, x),\Sigma(t,x))$ for some measurable function $a: [0,T] \times \mathbb{R} \to \mathbb{R}^m_+$.


The proposed auxiliary problem is
\begin{equation}
\label{eq:exploratory-problem-w-z}
\min_{\pi \in \Pi} \tilde{J}\bigl(\pi(\cdot);w\bigr)
:= \mathbb{E}[(x_T^\pi-w)^2] - (w-z)^2,
\quad \operatorname{s.t. } \quad \eqref{eq:exploratory-wealth-equation}.
\end{equation}

We highlight some important points on the auxiliary problem \eqref{eq:exploratory-problem-w-z} along with its relationship with the target problem
\eqref{eq:problem-w-z}.
First, for $\pi \in \Pi$, the distribution $\pi_t$ at any $t$ has support on the whole space instead of just the positive half plane.
Nevertheless, the resulting mean part of the policy $a_t := \int u\,\pi_t(u)\,du \geq 0$ always satisfies
the no-shorting constraint by construction. This represents a crucially different way of randomization compared to \cite{wangReinforcementLearningContinuous2020}. Therein, policy randomization is carried out on the action (control) space. Therefore, if the action space is constrained (i.e. truncated), the randomization is on a {\it truncated} space. By contrast, the randomization introduced here is on the {\it entire} space even if the action space is constrained.
Indeed, we can write  $\pi(t,x) = \mathcal{N}(a(t, x),\Sigma(t,x))=a(t,x)+\mathcal{N}(0,\Sigma(t,x))$ where $a$ is an {\it admissible} (no-shorting) policy for the target problem, namely, an admissible randomized policy is just to add some proper Gaussian noise to an originally admissible policy. Second, we will show that
applying RL we can solve the auxiliary problem \eqref{eq:exploratory-problem-w-z}, without having to know the model coefficients, to obtain an admissible Gaussian policy $\mathcal{N}(a(\cdot,\cdot);\Sigma(\cdot,\cdot))$.
Moreover, as one of the main results of this paper, we prove that under some suitable choice of the variance function $\Sigma(\cdot,\cdot)$,
   the mean  $a^*_t:=\int u\,\pi^*_t(u)\,du$ of the optimal Gaussian solution to the auxiliary problem, which automatically satisfies the no-shorting constraint,
   is optimal for the target problem.

\section{Analyzing the Auxiliary Problem}
In this section,
we present an explicit solution of the auxiliary problem \eqref{eq:exploratory-problem-w-z},
and compare it with that of the target problem \eqref{eq:problem-w-z},
in terms of the model parameters $\rho_t$ and $\sigma_t$.
A key is to exploit the freedom of choosing the variance function
$\Sigma(\cdot,\cdot)$ to relate the two problems.
These results establish the theoretical foundation for the model-free algorithm to be
introduced in the next section.

Define the optimal value function of \eqref{eq:exploratory-problem-w-z} by
\begin{equation*}
\tilde{V}(t,x;w) := \inf_{\pi \in \Pi}
\mathbb{E}[ (x^\pi_T - w)^2 - (w-z)^2 \mid x_t^\pi = x ].
\end{equation*}
The associated Hamiltonian is
\begin{equation*}
\begin{aligned}
H(t,x,\pi) &= \frac{\partial}{\partial x}\tilde{V}(t,x;w) \tilde{b}_t(\pi)
+ \frac{1}{2} \frac{\partial^2}{\partial x^2}\tilde{V}(t,x;w) \|\tilde{\sigma}_t(\pi)\|^2 \\
&= \int_{\mathbb{R}^m} \biggl(\frac{\partial}{\partial x}\tilde{V}(t,x;w) u^\intercal \sigma_t\rho_t + \frac{1}{2}
\frac{\partial^2}{\partial x^2}\tilde{V}(t,x;w) u^\intercal \sigma_t\sigma_t^\intercal u\biggr)\,\pi(u)\,du,
 \end{aligned}
\end{equation*}
where $(t,x,\pi) \in [0,T] \times \mathbb{R}\times \mathcal{P}_+(\mathbb{R}^m)$.
Hence, the HJB equation for \eqref{eq:exploratory-problem-w-z} is
\begin{equation}
\label{eq:pre-variance-regularized-HJB}
\left\{
\begin{aligned}
&\frac{\partial}{\partial t} v(t,x) + \inf_{\pi\in \mathcal{P}_+(\mathbb{R}^m)} \int \biggl(
\frac{\partial}{\partial x} v(t,x) \, u^\intercal \sigma_t \rho_t + \frac{1}{2}
\frac{\partial^2}{\partial x^2} v(t,x) u^\intercal \sigma_t \sigma_t^\intercal u \biggr)\, \pi(u)\, du = 0, \\
&v(T,x) = (x-w)^2 - (w-z)^2.
\end{aligned}
\right.
\end{equation}
However, for each $(t,x)$ the randomized control $\pi$ is of the form
$\pi = \mathcal{N}(a, \Sigma(t,x))$ with  the mean vector $a\in \mathbb{R}^m_+$ and a fixed, chosen $\Sigma(t,x)$. Substituting this form into \eqref{eq:pre-variance-regularized-HJB} we can further rewrite
the equation as
\begin{equation}
\label{eq:variance-regularized-HJB}
\left\{
\begin{aligned}
&\frac{\partial}{\partial t} v(t,x) + \inf_{a \in \mathbb{R}^m_+} \biggl(
\frac{\partial}{\partial x} v(t,x) \, a^\intercal \sigma_t \rho_t + \frac{1}{2}
\frac{\partial^2}{\partial x^2} v(t,x) a^\intercal \sigma_t \sigma_t^\intercal a \biggr) + \frac{1}{2} \frac{\partial^2}{\partial x^2} v(t,x)
\operatorname{tr}[\Sigma(t,x)\sigma_t \sigma_t^\intercal ]= 0, \\
&v(T,x) = (x-w)^2 - (w-z)^2.
\end{aligned}
\right.
\end{equation}

The solution of \eqref{eq:variance-regularized-HJB} depends on the specific choice of the variance function $\Sigma(\cdot,\cdot)$.
Note that when $\Sigma(\cdot,\cdot) \equiv 0$,
\eqref{eq:variance-regularized-HJB} reduces to the HJB equation for the target problem; see \citet{liDynamicMeanvariancePortfolio2002}.

Define $\mu^{ \Gamma_1 }_t,~\mu^{ \Gamma_2 }_t,~P^{ \Gamma_1 }_t,~P^{ \Gamma_2 }_t$ by
\begin{equation}
\label{eq:def-mu1-mu2-P1-P2}
\begin{aligned}
\mu^{ \Gamma_1 }_t &:= \operatorname*{argmin}_{\mu \in \mathbb{R}^m_+} \left[\frac{1}{2}\mu^\intercal \sigma_t\sigma_t^\intercal \mu - (\sigma_t\rho_t)^\intercal \mu\right],\\
\mu^{ \Gamma_2 }_t &:= \operatorname*{argmin}_{\mu \in \mathbb{R}^m_+}\left[ \frac{1}{2}\mu^\intercal \sigma_t\sigma_t^\intercal \mu + (\sigma_t\rho_t)^\intercal \mu\right],\\
P^{ \Gamma_1 }_t &:= \exp\biggl( \int_t^T [(\mu^{ \Gamma_1 }_s)^\intercal \sigma_s \sigma_s^\intercal(\mu^{ \Gamma_1 }_s) - 2 (\sigma_s\rho_s)^\intercal \mu^{ \Gamma_1 }_s]\,ds \biggr), \\
P^{ \Gamma_2 }_t &:= \exp\biggl( \int_t^T [(\mu^{ \Gamma_2 }_s)^\intercal \sigma_s \sigma_s^\intercal(\mu^{ \Gamma_2 }_s) + 2 (\sigma_s\rho_s)^\intercal \mu^{ \Gamma_2 }_s]\,ds \biggr).
\end{aligned}
\end{equation}
Clearly,
we have  $\mu_t^{\Gamma_1} \geq 0,~\mu_t^{\Gamma_2}\geq 0,~0 < P_t^{\Gamma_1} \leq 1$,
and $0 < P_t^{\Gamma_2} \leq 1$ for all $t$.

We now specify the  following \emph{state-dependent} variance function for exploration:
\begin{equation*}
\Sigma(t,x) = (x-w)^2\Psi_t,
\end{equation*}
where $\Psi:[0,T]\to\mathbb{S}_{++}^m$ is any given matrix-valued measurable function.

Define 
\begin{equation}
\label{eq:def-tildep1-tildep2}
\tilde{P}^{ \Gamma_1 }_t := {P}^{ \Gamma_1 }_t\exp\biggl( \int_t^T \operatorname{tr}[\Psi_s \sigma_s\sigma_s^\intercal] \,ds \biggr),
\qquad
\tilde{P}^{ \Gamma_2 }_t := {P}^{ \Gamma_2 }_t\exp\biggl( \int_t^T \operatorname{tr}[\Psi_s \sigma_s\sigma_s^\intercal] \,ds \biggr).
\end{equation}
We are going to show that
\begin{equation*}
\tilde{V}^{ \Gamma_1 }(t,x;w) := \tilde{P}_t^{ \Gamma_1 }(x-w)^2 - (w-z)^2
\end{equation*}
solves the HJB equation \eqref{eq:variance-regularized-HJB} in the region $\Gamma_1:=\{(t,x) \mid x-w < 0\}$,
and
\begin{equation*}
\tilde{V}^{ \Gamma_2 }(t,x;w) := \tilde{P}_t^{ \Gamma_2 }(x-w)^2 - (w-z)^2
\end{equation*}
solves the HJB equation \eqref{eq:variance-regularized-HJB} in the region $\Gamma_2:=\{(t,x) \mid x-w > 0\}$ respectively.

To this end, we first need the following scaling lemma.
\begin{lemma}
\label{lemma:a-scaling-lemma}
Let $\Sigma \in \mathbb{R}^{m \times m}, ~p \in \mathbb{R}^m$ and $s > 0$.
Assume that $\Sigma$ is strictly positive definite. Then,
\begin{equation*}
\min_{a \in \mathbb{R}^m_+} [\frac{1}{2} a^\intercal \Sigma a + s p^\intercal a] = s^2 \cdot
\min_{\mu \in \mathbb{R}^m_+} [\frac{1}{2} \mu^\intercal \Sigma \mu + p^\intercal \mu],
\end{equation*}
and
\begin{equation*}
\operatorname*{argmin}_{a \in \mathbb{R}^m_+} [\frac{1}{2} a^\intercal \Sigma a + s p^\intercal a] = s \cdot
\operatorname*{argmin}_{\mu \in \mathbb{R}^m_+} [\frac{1}{2} \mu^\intercal \Sigma \mu + p^\intercal \mu].
\end{equation*}
\end{lemma}

\begin{proof}
We apply the standard Karush--Kuhn--Tucker (KKT) condition.
Denote $f(a;s) = \frac{1}{2}a^\intercal \Sigma a + sp^\intercal a$.
Since $\Sigma$ is strictly positive definite, $f(\cdot;s)$ has a unique minimizer on
$\mathbb{R}^m_+$.

Consider the convex minimization problem with linear inequality constraints
\begin{equation}
\label{eq:appendix-min-f-a-s}
\min_{a \in \mathbb{R}^m}~~ f(a;s), \quad \operatorname{s.t.} \quad a \geq 0.
\end{equation}
Recall that a pair $(a, \lambda) \in \mathbb{R}^m_+ \times \mathbb{R}^m_+$ is said to be
a KKT pair of \eqref{eq:appendix-min-f-a-s} if
\begin{equation}
\label{eq:appendix-standard-KKT}
\lambda^\intercal a = 0, \qquad \lambda = \Sigma a + sp.
\end{equation}
By standard results in convex analysis
(see, e.g., \citealt[pp. 224--226]{boydConvexOptimization2023}),
a pair $(a, \lambda)$ is a KKT pair of \eqref{eq:appendix-min-f-a-s}
if and only of $a$ is its minimizer.

Fix $s > 0$. Let $(a_s,\lambda_s) \in \mathbb{R}^m_+ \times \mathbb{R}^m_+$ be a KKT pair of \eqref{eq:appendix-min-f-a-s}.
Then,
\begin{equation*}
\Bigl(\frac{\lambda_s}{s}\Bigr)^\intercal \Bigl(\frac{a_s}{s}\Bigr)=0,\qquad
\frac{\lambda_s}{s} = \Sigma \frac{a_s}{s} + p.
\end{equation*}
This shows that $(\frac{a_s}{s},\frac{\lambda_s}{s}) \in \mathbb{R}^m_+ \times \mathbb{R}^m_+$
is a KKT pair of \eqref{eq:appendix-min-f-a-s} at $s=1$. Hence,
\begin{equation*}
a_s = \operatorname*{argmin}_{a \in \mathbb{R}^m_+} f(a;s),\qquad
\frac{a_s}{s} = \operatorname*{argmin}_{a \in \mathbb{R}^m_+} f(a;1).
\end{equation*}
This proves the second claim of the lemma.
The first claim is obtained directly from the fact that $f(a_s,s) = s^2 f( \frac{a_s}{s},1)$.
\end{proof}

Now consider the region $\Gamma_1$ where $x - w < 0$.
Substituting $\frac{\partial^2}{\partial x^2}\tilde{V}^{ \Gamma_1 }(t,x;w) = 2\tilde{P}_t^{ \Gamma_1 } > 0$
and $\frac{\partial}{\partial x}\tilde{V}^{ \Gamma_1 }(t,x;w) = 2\tilde{P}_t^{ \Gamma_1 }(x-w)$ into the left-hand side of  \eqref{eq:variance-regularized-HJB} yields
\begin{equation*}
\begin{aligned}
\mathrm{LHS}
&= \dot{\tilde{P}}^{\Gamma_1}_t(x-w)^2 +2\tilde{P}^{\Gamma_1}_t \inf_{a\in\mathbb{R}^m_+} \biggl( \frac12 a^\intercal\sigma_t\sigma_t^\intercal a + (x-w)a^\intercal\sigma_t\rho_t \biggr) + \tilde{P}^{\Gamma_1}_t(x-w)^2\operatorname{tr}[\Psi_t\sigma_t\sigma_t^\intercal] \\
&= \dot{\tilde{P}}^{\Gamma_1}_t(x-w)^2 +2\tilde{P}^{\Gamma_1}_t \inf_{a\in\mathbb{R}^m_+} \biggl( \frac12 a^\intercal\sigma_t\sigma_t^\intercal a -|x-w|(\sigma_t\rho_t)^\intercal a \biggr) +\tilde{P}^{\Gamma_1}_t(x-w)^2\operatorname{tr}[\Psi_t\sigma_t\sigma_t^\intercal].
\end{aligned}
\end{equation*}
Applying Lemma \ref{lemma:a-scaling-lemma} with the scaling parameter $|x-w|$, we obtain
\begin{equation*}
\begin{aligned}
\mathrm{LHS}
= \biggl[ \dot{\tilde{P}}^{\Gamma_1}_t +\tilde{P}^{\Gamma_1}_t\Bigl( (\mu_t^{\Gamma_1})^\intercal\sigma_t\sigma_t^\intercal\mu_t^{\Gamma_1} -2(\sigma_t\rho_t)^\intercal\mu_t^{\Gamma_1} +\operatorname{tr}\{\Psi_t\sigma_t\sigma_t^\intercal\} \Bigr) \biggr](x-w)^2
= 0.
\end{aligned}
\end{equation*}
This verifies the solution in $\Gamma_1$. The case for $\Gamma_2$ is completely analogous and omitted here.

Next, we are to show that the optimal value function of the auxiliary problem is
\begin{equation}
\label{eq:solution-of-variance-regularized-HJB}
\tilde{V}(t,x;w) = \begin{cases}
\tilde{P}^{ \Gamma_1 }_t (x-w)^2 - (w-z)^2,&\quad \text{ if } x - w < 0,\\
\tilde{P}^{ \Gamma_2 }_t (x-w)^2 - (w-z)^2,&\quad \text{ if } x - w > 0,\\
-(w-z)^2,&\quad \text{ if } x - w = 0.
\end{cases}
\end{equation}
In fact, this function is continuously differentiable on  $\Gamma_3:=\{(t,x)\mid x=w\}$
because
\begin{equation}
\label{eq:continously-differentiable}
\begin{aligned}
\lim_{x\to w^-}\tilde{V}(t,x;w)&=\lim_{x\to w^+}\tilde{V}(t,x;w)=-(w-z)^2, \\
\lim_{x\to w^-} \frac{\partial}{\partial x} \tilde{V}(t,x;w)&=\lim_{x\to w^+} \frac{\partial}{\partial x} \tilde{V}(t,x;w)=0.
\end{aligned}
\end{equation}
It also clearly satisfies the boundary condition of the HJB equation \eqref{eq:variance-regularized-HJB} as $\tilde{P}^{\Gamma_1}_T=\tilde{P}^{\Gamma_2}_T=1$.
Furthermore,
\eqref{eq:solution-of-variance-regularized-HJB} satisfies \eqref{eq:variance-regularized-HJB} in the viscosity sense; see the proof of Theorem \ref{theorem:state-dependent} below and Appendix~\ref{appendix:definition-of-viscosity-solution} for the precise definition of viscosity solutions. Thus the standard uniqueness result of viscosity solution establishes that \eqref{eq:solution-of-variance-regularized-HJB} is the value function of the auxiliary problem.

It follows that we can obtain the optimal mean strategy $a^*$ by minimizing the Hamiltonian
\begin{equation*}
\begin{aligned}
a^*(t,x;w) &= \operatorname*{argmin}_{a \in \mathbb{R}^m_+}
\biggl( \frac{\partial}{\partial x}\tilde{V}(t,x;w) a^\intercal \sigma_t\rho_t + \frac{1}{2}
\frac{\partial^2}{\partial x^2}\tilde{V}(t,x;w) a^\intercal \sigma_t\sigma_t^\intercal a
\biggr) \\
&= \operatorname*{argmin}_{a \in \mathbb{R}^m_+}
\biggl(\frac{1}{2} a^\intercal \sigma_t\sigma_t^\intercal a + \frac{
\frac{\partial}{\partial x}\tilde{V}(t,x;w)}{ \frac{\partial^2}{\partial x^2}\tilde{V}(t,x;w)} a^\intercal \sigma_t\rho_t \biggr) \\
&= \operatorname*{argmin}_{a \in \mathbb{R}^m_+}
\biggl(\frac{1}{2} a^\intercal \sigma_t\sigma_t^\intercal a + (x-w) a^\intercal \sigma_t\rho_t \biggr) \\
&= \begin{cases}
-(x-w)\mu_t^{ \Gamma_1 },&\quad \text{ if } x - w < 0,\\
(x-w)\mu_t^{ \Gamma_2 },&\quad \text{ if } x - w \geq 0.
\end{cases}
\end{aligned}
\end{equation*}

The obtained optimal mean strategy $a^*$ turns out to be identical to the theoretically  optimal strategy $u^*$ of the target problem, derived  in \citet{liDynamicMeanvariancePortfolio2002}:
\begin{equation}
\label{eq:optimal-u-target}
u^*(t,x;w) = \begin{cases}
-(x-w)\mu_t^{ \Gamma_1 },&\quad \text{ if } x - w < 0,\\
(x-w)\mu_t^{ \Gamma_2 },&\quad \text{ if } x - w \geq 0.
\end{cases}
\end{equation}

We arrive at the following theorem.

\begin{theorem}
\label{theorem:state-dependent}
Let $\Psi:[0,T] \to \mathbb{S}^m_{++}$ be a matrix-valued measurable function, and
$\Pi$ be the policy space defined in \eqref{eq:exploratory-problem-w-z}
with the variance function set to be $\Sigma(t,x) = (x-w)^2\Psi_t$.
Then,
the HJB equation \eqref{eq:variance-regularized-HJB} has a unique viscosity solution \eqref{eq:solution-of-variance-regularized-HJB} which is also the optimal value function of the auxiliary problem \eqref{eq:exploratory-problem-w-z}.
Moreover, the optimal strategy of \eqref{eq:exploratory-problem-w-z} is
\begin{equation}
\pi^*(t,x;w) = \begin{cases}
\mathcal{N}\bigl(-(x-w)\mu_t^{ \Gamma_1 }, (x-w)^2\Psi_t\bigr),&\quad \text{ if } x - w < 0,\\
\mathcal{N}\bigl((x-w)\mu_t^{ \Gamma_2 }, (x-w)^2\Psi_t\bigr),&\quad \text{ if } x - w \geq 0,
\end{cases}
\end{equation}
where $\mu^{ \Gamma_1 }_t$ and $\mu^{ \Gamma_2 }_t$ are given by \eqref{eq:def-mu1-mu2-P1-P2}.
Furthermore,
the mean strategy $a^*(t,x;w) := \int_{\mathbb{R}^m} u \pi^*(u \mid t, x;w) \,du$
is optimal to the target problem \eqref{eq:problem-w-z}.
\end{theorem}

\begin{proof}
Let $\tilde{V}$ be the piecewise quadratic function defined by \eqref{eq:solution-of-variance-regularized-HJB}, which clearly satisfies
the terminal condition.

\emph{1. Verification in the interior regions $\Gamma_i$ ($i=1,2$).}
When $(t,x) \in \Gamma_i$,
we have
\begin{equation*}
\tilde{V}(t,x;w) = \tilde{P}^{\Gamma_i}_t (x-w)^2 - (w-z)^2.
\end{equation*}
Thus,
\begin{equation*}
\frac{\partial \tilde{V}}{\partial t} = \dot{\tilde{P}}^{\Gamma_i}_t (x-w)^2, \quad \frac{\partial \tilde{V}}{\partial x} = 2\tilde{P}^{\Gamma_i}_t (x-w), \quad \frac{\partial^2 \tilde{V}}{\partial x^2} = 2\tilde{P}^{\Gamma_i}_t.
\end{equation*}
Substituting these into the left-hand side of \eqref{eq:variance-regularized-HJB} yields
\begin{equation*}
\begin{aligned}
\text{LHS} &= \dot{\tilde{P}}^{\Gamma_i}_t (x-w)^2 + \inf_{a \in \mathbb{R}^m_+} \biggl\{ 2\tilde{P}^{\Gamma_i}_t (x-w) a^\intercal \sigma_t \rho_t + \tilde{P}^{\Gamma_i}_t a^\intercal \sigma_t \sigma_t^\intercal a \biggr\} + \tilde{P}^{\Gamma_i}_t (x-w)^2\text{tr}\{\Psi_t \sigma_t \sigma_t^\intercal\} \\
&= (x-w)^2 [ \dot{\tilde{P}}^{\Gamma_i}_t + \tilde{P}^{\Gamma_i}_t \text{tr}\{\Psi_t \sigma_t \sigma_t^\intercal\}] + 2\tilde{P}^{\Gamma_i}_t \inf_{a \in \mathbb{R}^m_+} \biggl\{ (x-w) a^\intercal \sigma_t \rho_t + \frac{1}{2} a^\intercal \sigma_t \sigma_t^\intercal a \biggr\}.
\end{aligned}
\end{equation*}
Since $\tilde{P}^{\Gamma_i}_t = P^{\Gamma_i}_t \exp\bigl( \int_t^T \text{tr}\{\Psi_s \sigma_s \sigma_s^\intercal\} \, ds \bigr)$,
we have
\begin{equation*}
\dot{\tilde{P}}^{\Gamma_i}_t + \tilde{P}^{\Gamma_i}_t \text{tr}\{\Psi_t \sigma_t \sigma_t^\intercal\}
= \dot{P}^{\Gamma_i}_t \exp\biggl( \int_t^T \text{tr}\{\Psi_s \sigma_s \sigma_s^\intercal\} \, ds \biggr)
= \dot{P}^{\Gamma_i}_t \frac{\tilde{P}^{\Gamma_i}_t}{P^{\Gamma_i}_t}.
\end{equation*}
Therefore,
\begin{equation*}
\text{LHS} = \frac{\tilde{P}^{\Gamma_i}_t}{P^{\Gamma_i}_t} \biggl[ \dot{P}^{\Gamma_i}_t (x-w)^2 + 2P^{\Gamma_i}_t \inf_{a \in \mathbb{R}^m_+} \biggl\{ (x-w) a^\intercal \sigma_t \rho_t + \frac{1}{2} a^\intercal \sigma_t \sigma_t^\intercal a \biggr\} \biggr].
\end{equation*}
The term inside the brackets is exactly the HJB equation for the target problem
and thus equals zero.
This verifies that $\tilde{V}$ satisfies the HJB equation in the regions $\Gamma_i$ ($i=1,2$).

\emph{2. Verification on the boundary $\Gamma_3$.}
\eqref{eq:continously-differentiable} ensures that both $\tilde{V}$ and $\frac{\partial}{\partial x}\tilde{V}$ at points on $\Gamma_3$.
Moreover,
we have
\begin{equation*}
\begin{aligned}
\lim_{x \to w^-} \frac{\partial^2}{\partial x^2} \tilde{V}(t,x;w) &= \tilde{P}^{\Gamma_1}_t,\\
\lim_{x \to w^+} \frac{\partial^2}{\partial x^2} \tilde{V}(t,x;w) &= \tilde{P}^{\Gamma_2}_t,
\end{aligned}
\end{equation*}
implying that for all $(t,x) \in \Gamma_3$,
\begin{equation*}
\begin{aligned}
D^{1,2,+}_{t,x}\tilde{V}(t,x) &= \{ 0 \} \times \{0\} \times [\tilde{P}^{\Gamma_2}_t, \infty),\\
D^{1,2,-}_{t,x}\tilde{V}(t,x) &= \{ 0 \} \times \{0 \} \times (-\infty, \tilde{P}^{\Gamma_1}_t].
\end{aligned}
\end{equation*}
Therefore,
for any $(t,x) \in \Gamma_3$ and $(q,p,P) \in D^{1,2,+}_{t,x}\tilde{V}(t,x)$,
\begin{equation*}
q + \inf_{a \geq 0} G(t,x,a,p,P) \geq 0 + \inf_{a \geq 0}G(t,x,a,0,\tilde{P}^{\Gamma_2}_t) = 0,
\end{equation*}
and
for any $(t,x) \in \Gamma_3$ and $(q,p,P) \in D^{1,2,-}_{t,x}\tilde{V}(t,x)$,
\begin{equation*}
q + \inf_{a \geq 0} G(t,x,a,p,P) \leq 0 + \inf_{a \geq 0}G(t,x,a,0,\tilde{P}^{\Gamma_1}_t) = 0.
\end{equation*}
This establishes that $\tilde V$ is a viscosity solution of \eqref{eq:variance-regularized-HJB}. The uniqueness of the viscosity solution is standard; see
e.g. \citet{yongStochasticControlsHamiltonian1999}.
\end{proof}

Several remarks are in order.
First off, in the next section we will show that the randomized auxiliary problem \eqref{eq:exploratory-problem-w-z} with unknown parameters can be solved algorithmically using the recently developed continuous-time RL theory. Theorem \ref{theorem:state-dependent}, on the other hand, indicates that the mean of the optimal Gaussian policy recovers the solution to the original target problem \eqref{eq:problem-w-z}. Crucially, because the mean of any admissible Gaussian policy is nonnegative by construction, it satisfies the no-shorting constraint of \eqref{eq:problem-w-z}.
Second, the function $\Psi_t$ does not affect the optimal mean policy;
it controls only the exploration level via the variance of the underlying Gaussian policy.
Finally, both the mean and the variance of $\pi^*(t,x;w)$ vanish as $x$ approaches $w$, and at $x=w$ the policy degenerates to Dirac at 0 (i.e. completely no risky allocation) and, as a result, the (discounted) wealth stays put at $x=w$. In other words,
$x = w$ is a ``dividing boundary'' of the state space under $\pi^*$,
preventing any admissible wealth trajectory from crossing between $\Gamma_1$ and $\Gamma_2$.

\section{Solving Portfolio Selection with Unknown Dynamics}
In this section, we develop a model-free algorithm to learn the optimal policy of the target portfolio selection problem \eqref{eq:problem-w-z} without knowing the investment opportunity set.
Based on the theoretical results presented in the previous section,
it is equivalent to searching for the solution  of the auxiliary, exploratory problem \eqref{eq:exploratory-problem-w-z}.
To do so we apply a trajectory-based (i.e. data-driven) policy gradient estimator and update our policy within the feasible set via the proximal gradient descent method.

\subsection{Parameterization and Policy Evaluation}
We take $\Sigma(t,x) = (x-w)^2\Psi_0$ where $\Psi_0 \in \mathbb{S}^n_{++}$
is a fixed matrix.
In implementation, one can choose  $\Psi_0=\lambda I_m$ where $\lambda$ is a tuning hyperparameter controlling the exploration strength.
Consider the following parameterized policy:
\begin{equation}
\label{eq:parameterized-pi-phi}
\pi^\phi(t,x;w) = \mathcal{N}\bigl(u^\phi(t,x;w), (x-w)^2\Psi_0\bigr),
\end{equation}
where $u^\phi(t,x;w)$ is a parameterized function that takes value in $\mathbb{R}^m_+$.

In general, one can use feedforward neural networks to represent  $u^\phi$ and set the last
layer to a ReLU operation to ensure the nonnegativity (see, e.g., \citealt{bishopDeepLearningFoundations2024}).
In this paper,
however,
we use a piecewise linear function to represent $u^\phi$, inspired by the theoretical structure of
the optimal control \eqref{eq:optimal-u-target}. Specifically, the parameterized function
is set to be
\begin{equation}
\label{eq:parameterized-mean-policy}
u^\phi(t,x;w) = \begin{cases}
-(x-w) \phi_t^{(1)},&\quad \text{ if } x - w < 0,\\
(x-w) \phi_t^{(2)},&\quad \text{ if } x - w \geq 0.
\end{cases}
\end{equation}
For numerical implementation, the time interval $[0,T]$ is discretized into a finite grid
$\{i\frac{T}{N},~~i=0,1,2,\ldots N\}$. Let the values of the functions $\phi^{(1)}$ and $\phi^{(2)}$ at the time instant $t_i = i \frac{T}{N}$ be
$\phi_i^{(1)} \in \mathbb{R}^m_+$  and $\phi_i^{(2)} \in \mathbb{R}^m_+$ respectively. Hence, there are $2N+2$ vectors to be optimized.\footnote{If the portfolio selection problem is  time invariant, then they reduce to two vectors, as the coefficients are identical for all $t_i$.}
With a slight abuse of notation, henceforth we use $\phi$ to denote both the function $\phi:=(\phi^{(1)},\phi^{(2)})$ and the sequence of vectors $\phi:=\{\phi^{(1)}_i,\phi^{(2)}_i\}_{i=0}^N$.

The value function (i.e. cost function) of $\pi^\phi$ is 
\begin{equation}
\label{eq:value-function-parameterized-policy}
{V}^\phi(t,x;w) = \mathbb{E}[(x^\phi_T - w)^2 - (w-z)^2 \mid x^\phi_t = x]
\end{equation}
where $x^\phi$ is the solution of the exploratory dynamics \eqref{eq:exploratory-wealth-equation} under $\pi^\phi$.
The piecewise linear policy \eqref{eq:parameterized-mean-policy} suggests
an explicit way to parameterize ${V}^\phi$. Indeed,
by the Feynman--Kac formula,
the right hand side of  \eqref{eq:value-function-parameterized-policy} satisfies the PDE
\begin{equation*}
\left\{
\begin{aligned}
&\frac{\partial}{\partial t} v(t,x) +
\frac{\partial}{\partial x} v(t,x) \, (\sigma_t \rho_t)^\intercal u^\phi(t,x;w) + \frac{1}{2}
\frac{\partial^2}{\partial x^2} v(t,x) (u^\phi(t,x;w))^\intercal \sigma_t \sigma_t^\intercal u^\phi(t,x;w) \\
&\hspace{8em} + \frac{1}{2} \frac{\partial^2}{\partial x^2} v(t,x)
\operatorname{tr}[\Psi_0\sigma_t \sigma_t^\intercal ](x-w)^2= 0, \\
&v(T,x) = (x-w)^2 - (w-z)^2.
\end{aligned}
\right.
\end{equation*}
Solving it yields
\begin{equation}
\label{eq:theoretical-V-phi}
{V}^\phi(t,x;w) = \begin{cases}
e^{\theta^{ \Gamma_1 }_t}(x-w)^2 - (w-z)^2, &\quad \text{ if } x - w < 0,\\
e^{\theta^{ \Gamma_2} _t}(x-w)^2 - (w-z)^2, &\quad \text{ if } x - w > 0,\\
-(w-z)^2, &\quad \text{ if } x - w = 0,
\end{cases}
\end{equation}
where $\theta^{ \Gamma_1 }_t$ and $\theta^{ \Gamma_2 }_t$ are given, theoretically,  by
\begin{equation}
\label{eq:theoretical-theta-values}
\begin{aligned}
\theta^{ \Gamma_1 }_t &:= \int_t^T\biggl( (\phi^{(1)}_s)^\intercal \sigma_s\sigma_s^\intercal (\phi^{(1)}_s) - 2(\sigma_s\rho_s)^\intercal \phi^{(1)}_s
 + \operatorname{tr}[\Psi_0\sigma_s\sigma_s^\intercal]\biggr)\,ds,\\
\theta^{ \Gamma_2 }_t &:= \int_t^T\biggl( (\phi^{(2)}_s)^\intercal \sigma_s\sigma_s^\intercal (\phi^{(2)}_s) + 2(\sigma_s\rho_s)^\intercal \phi^{(2)}_s
 + \operatorname{tr}[\Psi_0\sigma_s\sigma_s^\intercal]\biggr)\,ds.
\end{aligned}
\end{equation}


We do not know the functions $\theta^{ \Gamma_1 }$ and $\theta^{ \Gamma_2 }$  due to the unknown model parameters; so we parameterize them.
Following the parameterization used for $\phi^{(1)}$ and $\phi^{(2)}$, we discretize the trajectories
$\{\theta^{ \Gamma_1 }_t,\theta^{ \Gamma_2 }_t\}_{0 \leq t \leq T}$ using $2N+2$ scalars $\{\theta^{(1)}_i,\theta^{(2)}_i\}_{i=0}^N$.
This is equivalent to  the parameterization of the value function in the form
\begin{equation}
\label{eq:parameterized-value-function}
V^\theta(t,x;w) = \begin{cases}
e^{\theta^{ (1) }_t}(x-w)^2 - (w-z)^2, &\quad \text{ if } x - w < 0,\\
e^{\theta^{ (2)} _t}(x-w)^2 - (w-z)^2, &\quad \text{ if } x - w > 0,\\
-(w-z)^2, &\quad \text{ if } x - w = 0.
\end{cases}
\end{equation}

As with $\phi$, we use $\theta$ to denote interchangeably both the function $\theta:=(\theta^{(1)},\theta^{(2)})$ and the sequence of vectors $\theta:=\{\theta^{(1)}_i,\theta^{(2)}_i\}_{i=0}^N$.
We now apply the martingale loss method proposed in \citet{jiaPolicyEvaluationTemporaldifference2022} to perform the policy evaluation, in order to learn $(\theta^{(1)},\theta^{(2)})$ or equivalently the vectors  $\{\theta^{(1)}_i,\theta^{(2)}_i\}_{i=0}^N$ under a given $\pi^\phi$.
Let $(x(\cdot),u(\cdot))$ be the wealth--portfolio pair of \eqref{eq:wealth-equation} while $u(\cdot)$ is sampled from the randomized policy $\pi^\phi(\cdot)$.
The martingale loss is defined by
\begin{equation*}
L(\theta) := \mathbb{E}\int_0^T
\bigl(V^\theta(t,x_t;w)-[(x_T-w)^2 - (w-z)^2]\bigr)^2\,dt.
\end{equation*}
We use the mini-batch stochastic gradient descent (SGD) method to optimize the martingale loss.
Discretizing $(x(\cdot),u(\cdot))$ into $\{(x_i, u_i)\}_{i=0}^N$ on the given time grid, we can compute the
gradient by
\begin{equation}
\label{eq:martingale-loss-gradient-formula}
\begin{aligned}
\frac{\partial}{\partial \theta^{(1)}_i} L(\theta)
&= \begin{cases}
\mathbb{E}\biggl[2\Bigl(
V^\theta(t_i, x_i; w) - \bigl[(x_N - w)^2 - (w - z)^2\bigr]
\Bigr) \cdot e^{\theta^{ (1) }_i}(x_i-w)^2 \frac{T}{N}\biggr],
&\quad\text{ if } x_i - w < 0,\\
0, &\quad\text{ if } x_i - w \geq 0,
\end{cases}\\
\frac{\partial}{\partial \theta^{(2)}_i} L(\theta)
&= \begin{cases}
0,&\quad\text{ if } x_i - w < 0,\\
\mathbb{E}\biggl[2\Bigl(
V^\theta(t_i, x_i; w) - \bigl[(x_N - w)^2 - (w - z)^2\bigr]
\Bigr) \cdot e^{\theta^{ (2) }_i}(x_i-w)^2 \frac{T}{N} \biggr],
&\quad\text{ if } x_i - w \geq 0.
\end{cases}
\end{aligned}
\end{equation}

As discussed earlier,
the policy parameter $\phi$ is identified  as vectors $\{(\phi^{(1)}_i,\phi^{(2)}_i)\}_{i=0}^N$ for \eqref{eq:parameterized-mean-policy},
and the value function parameter $\theta$  as vectors $\{(\theta^{(1)}_i,\theta^{(2)}_i)\}_{i=0}^N$ for \eqref{eq:parameterized-value-function}.
In our implementation,
$\phi^{(1)}_N$ and $\phi^{(2)}_N$ are fixed to be 0 as they are never used
(the system exits at the terminal time and the terminal control will not be applied).
In addition,
$\theta^{(1)}_N$ and $\theta^{(2)}_N$ are fixed to be 0 to satisfy the terminal condition
$V^\theta(T,x) = (x-w)^2 - (w-z)^2$.

\subsection{Policy Gradient}
Policy gradient is a popular approach in model-free RL.
It uses the learned value function to construct an estimator of its gradient with respect to the policy parameter, providing the direction of moving to the next policy iterate.
There are many such estimators proposed,
with different trade-offs in terms of efficiency, bias, and variance.
In this paper,
we use the gradient estimator of \citet{jiaPolicyGradientActorcritic2022} developed for
continuous-time RL and hence well suited for the setting of this paper.

\begin{lemma}
\label{lemma:policy-gradient-phi-general}
Let $\pi^\phi$ be a parameterized randomized policy, and let
\[
V^\phi(t,x)
:=
\mathbb{E}\!\left[
g(x_T)\mid x_t=x
\right]
\]
be the corresponding value function,
where $(x(\cdot),u(\cdot))$ denotes the wealth--portfolio trajectory generated by
$\pi^\phi$.
Then the policy gradient admits the representation
\begin{equation}
\label{eq:policy-gradient-phi-general}
\frac{\partial}{\partial\phi}
V^\phi(t,x)
=
\mathbb{E}
\left[
\int_t^T
\left(
\frac{\partial}{\partial\phi}
\log
\pi^\phi(u_s\mid s,x_s)
\right)
\,dV^\phi(s,x_s)
\Biggm|
x_t=x
\right].
\end{equation}
\end{lemma}

\begin{proof}
The result follows directly from Theorem~5 of \citet{jiaPolicyGradientActorcritic2022}.
\end{proof}

Applying Lemma~\ref{lemma:policy-gradient-phi-general} to the exploratory objective
$\tilde J(\pi^\phi(\cdot);w)$ yields
\begin{equation}
\frac{\partial }{\partial \phi}\tilde{J}(\pi^\phi(\cdot);w)
= \frac{\partial}{\partial \phi} V^\phi(0,x;w)
= \mathbb{E}\int_0^T
\biggl(
\frac{\partial}{\partial \phi}
\log \pi^\phi(u_t \mid t, x_t)
\biggr)
\,dV^\phi(t,x_t;w).
\end{equation}
Recall that  $\pi^\phi$ is constructed by \eqref{eq:parameterized-pi-phi} and \eqref{eq:parameterized-mean-policy},
whereas its parameter $\phi$ is taken as $\{(\phi^{(1)}_i,\phi^{(2)}_i)\}_{i=0}^N$.
Given the discretized wealth--portfolio trajectory $\{(x_i, u_i)\}_{i=0}^N$,
the policy gradient \eqref{eq:policy-gradient-phi-general} in each component is
\begin{equation}
\label{eq:policy-gradient-formula}
\begin{aligned}
\frac{\partial}{\partial \phi^{(1)}_i} \tilde{J}(\pi^\phi(\cdot);w)
&= \begin{cases}
\mathbb{E}\biggl[-\Psi_0^{-1}\Bigl( \frac{u_i}{x_i - w} + \phi_i^{(1)}\Bigr)
\Bigl(V^\phi( t_{i+1},x_{i+1};w) - V^\phi( t_i,x_{i};w)\Bigr)
\biggr],
&\quad\text{ if } x_i - w < 0,\\
0, &\quad\text{ if } x_i - w \geq 0,
\end{cases}\\
\frac{\partial}{\partial \phi^{(2)}_i} \tilde{J}(\pi^\phi(\cdot);w)
&= \begin{cases}
0,&\quad\text{ if } x_i - w < 0,\\
\mathbb{E}\biggl[\Psi_0^{-1}\Bigl( \frac{u_i}{x_i - w} - \phi_i^{(2)}\Bigr)
\Bigl(V^\phi( t_{i+1},x_{i+1};w) - V^\phi( t_i,x_{i};w)\Bigr)
\biggr],
&\quad\text{ if } x_i - w \geq 0.
\end{cases}
\end{aligned}
\end{equation}
To satisfy the nonnegativity constraint on the mean of
$\pi^\phi$,
we ensure $\phi^{(1)}_i \geq 0$ and $\phi^{(2)}_i \geq 0$ by projecting them onto  $\mathbb{R}^m_+$ after each gradient descent step.
This is known as the proximal gradient method for constrained optimization problems \citep{beckFirstorderMethodsOptimization2017}.
In our case,
the proximal gradient method is guaranteed to converge because $V^\phi(0,x;w)$
is convex in $\phi$; see \eqref{eq:theoretical-V-phi}--\eqref{eq:theoretical-theta-values}.

Finally,
the update rule of the multiplier $w$ is based on 
applying the standard stochastic approximation method to the equation
\eqref{eq:def-w-star},
where $\mathbb{E}[x^*(T;w)]$ is estimated as the empirical mean over wealth processes under portfolios sampled from $\pi^\phi$.

We summarize the overall procedure as Algorithm \ref{algo:policy-gradient}.
\begin{algorithm}[H]
\caption{\label{algo:policy-gradient}Model-Free Policy Gradient for the Mean--Variance Problem}
\begin{algorithmic}[1]
\REQUIRE Initial parameters $\phi,~\theta,~w$; learning rates $\alpha_\phi,~\alpha_\theta,~\alpha_w$; batch size $M$; iteration counts $n_{\phi},~n_{\theta},n_w$, and $n_{\text{warmup}}$.
\FOR{$\ell=1$ \TO $n_{\text{warmup}}$}
  \STATE Sample $M$ trajectories by executing $\pi^\phi$,
organized as $\mathcal{D}_\ell:=\{(x^{(j)}_i,u^{(j)}_i),~i=0,1,\ldots,N\}_{j=1}^{M}$.
  \STATE Estimate the gradient $g_\theta \leftarrow \partial_\theta L(\theta)$ using $\mathcal{D}_\ell$ via \eqref{eq:martingale-loss-gradient-formula} and update
  \begin{equation*}
    \theta \leftarrow \theta - \alpha_{\theta} g_{\theta}.
  \end{equation*}
\ENDFOR
\STATE Initialize an empty data buffer $\mathcal{X}$.
\FOR{$k = 1$ \TO $n_{\phi}$}
  \STATE Initialize an empty data buffer $\mathcal{D}$.
  \FOR{$\ell = 1$ \TO $n_{\theta}$}
    \STATE Sample $M$ trajectories by executing $\pi^\phi$,
  organized as $\mathcal{D}_\ell:=\{(x^{(j)}_i,u^{(j)}_i),~i=0,1,\ldots,N\}_{j=1}^{M}$.
    \STATE Estimate the gradient $g_\theta \leftarrow \partial_\theta L(\theta)$ using $\mathcal{D}_\ell$ via \eqref{eq:martingale-loss-gradient-formula} and update
    \begin{equation*}
      \theta \leftarrow \theta - \alpha_{\theta} g_{\theta}.
    \end{equation*}
    \STATE Append $\mathcal{D}_\ell$ to $\mathcal{D}$.
  \ENDFOR
  \STATE Estimate the policy gradient $g_\phi \leftarrow \partial_\phi \tilde{J}(\pi^\phi)$ using $\mathcal{D}$ via \eqref{eq:policy-gradient-formula} and update
  \begin{equation*}
    \phi \leftarrow \max(0, \phi - \alpha_\phi g_\phi).
  \end{equation*}
  \STATE Compute the empirical mean of $x_T$ in $\mathcal{D}$ and append it to $\mathcal{X}$.
  \IF{$(k \text{ mod } n_w) = 0$}
    \STATE Compute $\bar{x}_T$ as the average of the $n_w$ most recent entries in $\mathcal{X}$ and update
      \begin{equation*}
        w \leftarrow w - \alpha_w(\bar{x}_T - z).
      \end{equation*}
  \ENDIF
\ENDFOR
\RETURN Updated policy $\pi^\phi$ and updated multiplier $w$.
\end{algorithmic}
\end{algorithm}

\section{A Numerical Example}
We test Algorithm \ref{algo:policy-gradient} on a time-invariant example of the problem \eqref{eq:mv-problem} under no-shorting constraints, and then
compare the results with the theoretical (ground truth) solution obtained from \eqref{eq:def-mu1-mu2-P1-P2} and  \eqref{eq:optimal-u-target}.

\textbf{Problem data and theoretical solution.}
We set
\[
b - r \mathbf{1} =
\begin{bmatrix}
0.06 \\ 0.12 \\ 0.16
\end{bmatrix},\quad
\sigma =
\begin{bmatrix}
0.20&0&0\\
0&0.25&0\\
0&0&0.35
\end{bmatrix}
\begin{bmatrix}
1 & 0 & 0\\
\frac12 & \sqrt3/2 & 0\\
\frac12 & \sqrt3/6 & \sqrt6/3
\end{bmatrix}.
\]
Note here $\sigma=\operatorname{diag}\{\nu\}\Lambda$,
where $\nu$ is the asset volatilities and $\Lambda$ is the Cholesky factor of the correlation matrix.
Moreover, set the investment horizon $T=1$ and
the initial wealth $x_0=1$.

In our experiment, we take the exploration variance parameter $\Psi_0=\tfrac{1}{5}I_3$,
the time discretization step size $\frac{T}{N}=0.01$,
and the batch size $M=512$ for each simulation.

The theoretical solution \eqref{eq:optimal-u-target} of the no-shorting MV problem \eqref{eq:mv-problem} in this case specializes to
\begin{equation*}
u^*(t,x) = \begin{cases}
-\phi^*(x - w^*),&\quad \text{ if } x < w^*,\\
0,&\quad\text{ if } x \geq w^*,
\end{cases}
\end{equation*}
where the pair $(\phi^*,w^*)$ can be computed  from solving the following problem
\begin{equation*}
\max_{w \in \mathbb{R}} \min_{\phi \in \mathbb{R}_+^3} J(\phi,w) := e^{[\phi^\intercal \sigma\sigma^\intercal \phi - 2(\sigma\rho)^\intercal \phi]T}(x_0 - w)^2 - (w - z)^2.
\end{equation*}
Note
$\phi^*$ is exactly the coefficient $\mu_t^{\Gamma_1}$ defined in \eqref{eq:def-mu1-mu2-P1-P2},
which is time-independent in this example.
Moreover,
$\mu_t^{\Gamma_2}\equiv0$ as $\sigma_t\rho_t=b_t - r_t\mathbf{1} \geq 0$ for all $t$.

As the dynamics is time-invariant,
the policy parameter $\phi$ reduces to two constant vectors $\phi^{(1)},\phi^{(2)}\in\mathbb{R}^3_+$.
Moreover,
because the initial $x_0$ lies in either $\Gamma_1$ or $\Gamma_2$,
and as explained earlier the state process driven by the parameterized policy \eqref{eq:parameterized-pi-phi}--\eqref{eq:parameterized-mean-policy} will never escape from the initial region,
it suffices to parameterize and update only one of these vectors,
which we denote by $\phi$.

We run Algorithm 1 to generate iterates of $(\phi,w) \in  \mathbb{R}^3_+ \times \mathbb{R}$,
and compare the final estimates with the corresponding theoretical values.
We emphasize that Algorithm 1 is completely model-free,
which uses policy gradient to update $\phi$ and stochastic approximation to update $w$.

\textbf{Results for different expected returns.}
Figure~\ref{fig:results-for-z} shows the iterates of $\phi$, $w$, $\mathbb{E}[x_T]$,
and $\operatorname{Var}(x_T)$ for $z=1.10$, $1.15$, and $1.20$ respectively.
Table~\ref{tab:results-for-z} reports the final estimates.
These results demonstrate that our algorithm leads to convergence  to the ground truth values without the knowledge of model parameters. In particular, it learns the optimal values of the terminal mean and variance remarkably well.

\textbf{Extreme target return.}
We next consider an extreme case $z=2.0$, corresponding to a 100\% expected return within a year.
Figure~\ref{fig:results-for-z-extreme} and Table~\ref{tab:results-for-z} summarize the results.
Note that Algorithm~1 enforces the constraint $\phi \geq 0$ using the projected gradient descent method,
ensuring that the mean strategy $u^\phi$ \eqref{eq:parameterized-mean-policy} always satisfies the no-shorting constraint.
In this case, in terms of the finally learned values, $\phi$ is significantly away from the ground truth, while the other parameters, especially the terminal mean and variance,  are still very close to the respective theoretical values. In other words, the portfolio policy learned from our algorithm  is quite different from the theoretically optimal one, but it is nearly optimal because the resulting performance is close to that of the oracle. One must note that in any reinforcement learning problem, what matters is the achieved  {\it performance} much more than the proximity  of the learned policy from the optimal one. On the other hand, the high target return renders
a very high variance (risk), consistent with the classical mean--variance theory.

\textbf{Robustness check on initialization and learning rate.}
The above reports the numerical results with initial values $(\theta,\phi,w)=(0,(0.1,0.1,0.1),z)$,
and learning rates $(\alpha_\theta,\alpha_\phi,\alpha_w)=(0.001,0.1,0.01)$.
We have also conducted experiments with different initial points and different learning rates.
We do not include the corresponding figures and tables here for a concise numerical section, but just report that the algorithm is robust to initialization.
For learning rates,
the observed behavior follows the usual pattern of SGD:
large learning rates give faster convergence but larger variance,
while smaller learning rates lead to slower convergence but smaller variance.

\begin{figure}[p]
\centering
\includegraphics[width=.9\textwidth]{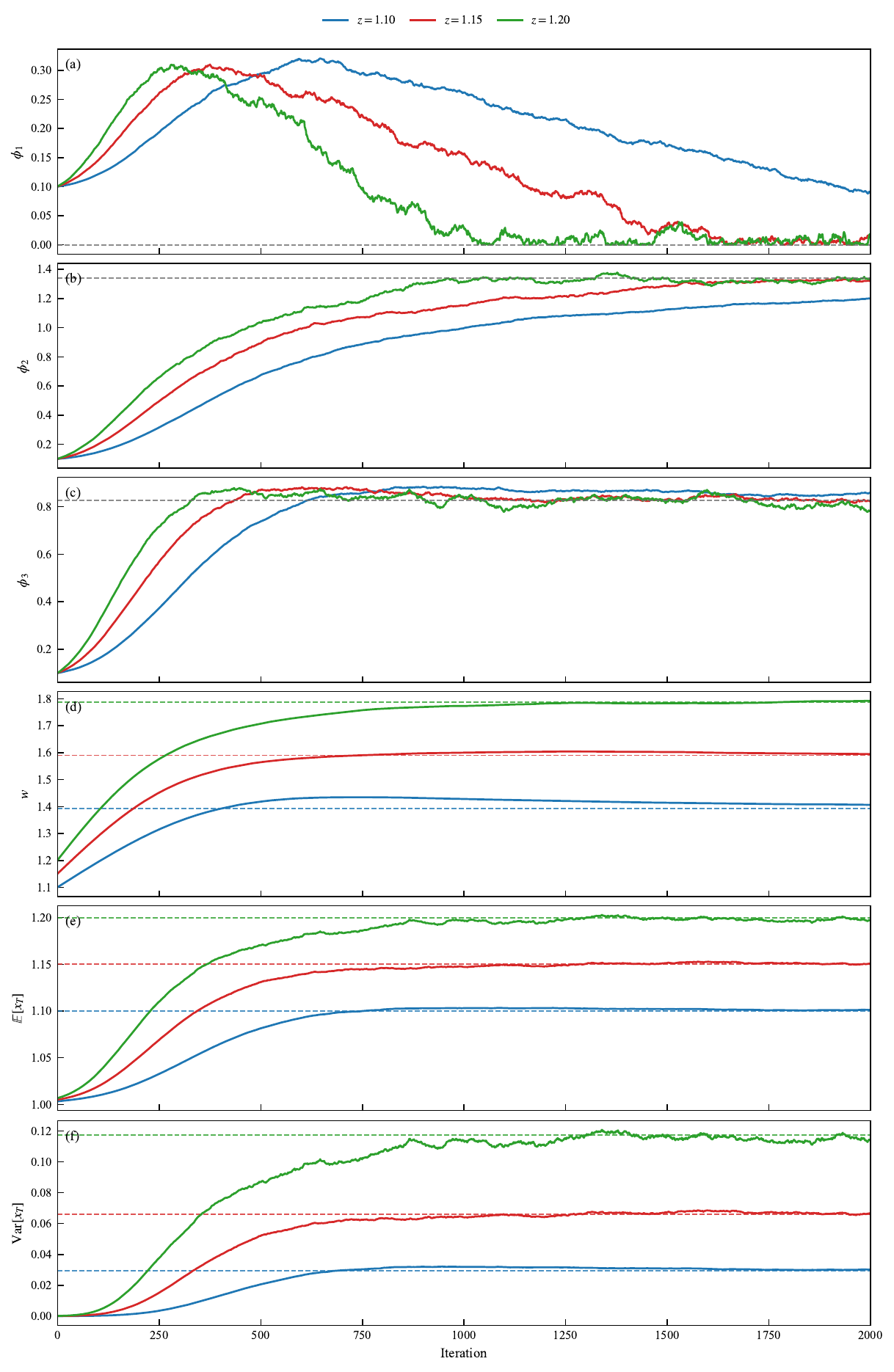}
\caption{\label{fig:results-for-z}
Iterates of $\phi$, $w$, $\mathbb{E}[x_T]$,
and $\operatorname{Var}(x_T)$ in the main loop of Algorithm~1 for different expected returns $z$.
For clarity,
the first three panels depict the three components of $\phi$ respectively.
Dashed horizontal lines indicate the corresponding theoretical values.
}
\end{figure}

\begin{table}[h]
\caption{Final estimates for different expected returns $z$,
including the extreme case $z=2.0$.}
\label{tab:results-for-z}
\centering
\small
\setlength{\tabcolsep}{4pt}
\begin{tabular}{ccccccccccc}
\toprule
$z$ & $\phi$ & $\phi^*$ & $w$ & $w^*$ & $\mathbb{E}[x_T]$ & $\mathbb{E}[x_T^*]$ & $\operatorname{Var}(x_T)$ & $\operatorname{Var}(x_T^*)$ \\
\midrule
1.10 & $(0.090, 1.199, 0.858)$ & $(0.000,\,1.341,\,0.827)$ & 1.406 & 1.393 & 1.101 & 1.100 & 0.030 & 0.029 \\
1.15 & $(0.015, 1.321, 0.823)$ & $(0.000,\,1.341,\,0.827)$ & 1.595 & 1.590 & 1.150 & 1.150 & 0.066 & 0.066 \\
1.20 & $(0.017, 1.338, 0.781)$ & $(0.000,\,1.341,\,0.827)$ & 1.792 & 1.787 & 1.197 & 1.200 & 0.114 & 0.117 \\
2.00 & $(0.189, 1.010, 0.898)$ & $(0.000,\,1.341,\,0.827)$ & 4.965 & 4.934 & 1.957 & 2.000 & 2.742 & 2.934 \\
\bottomrule
\end{tabular}
\end{table}

\begin{figure}[p]
\centering
\includegraphics[width=.9\textwidth]{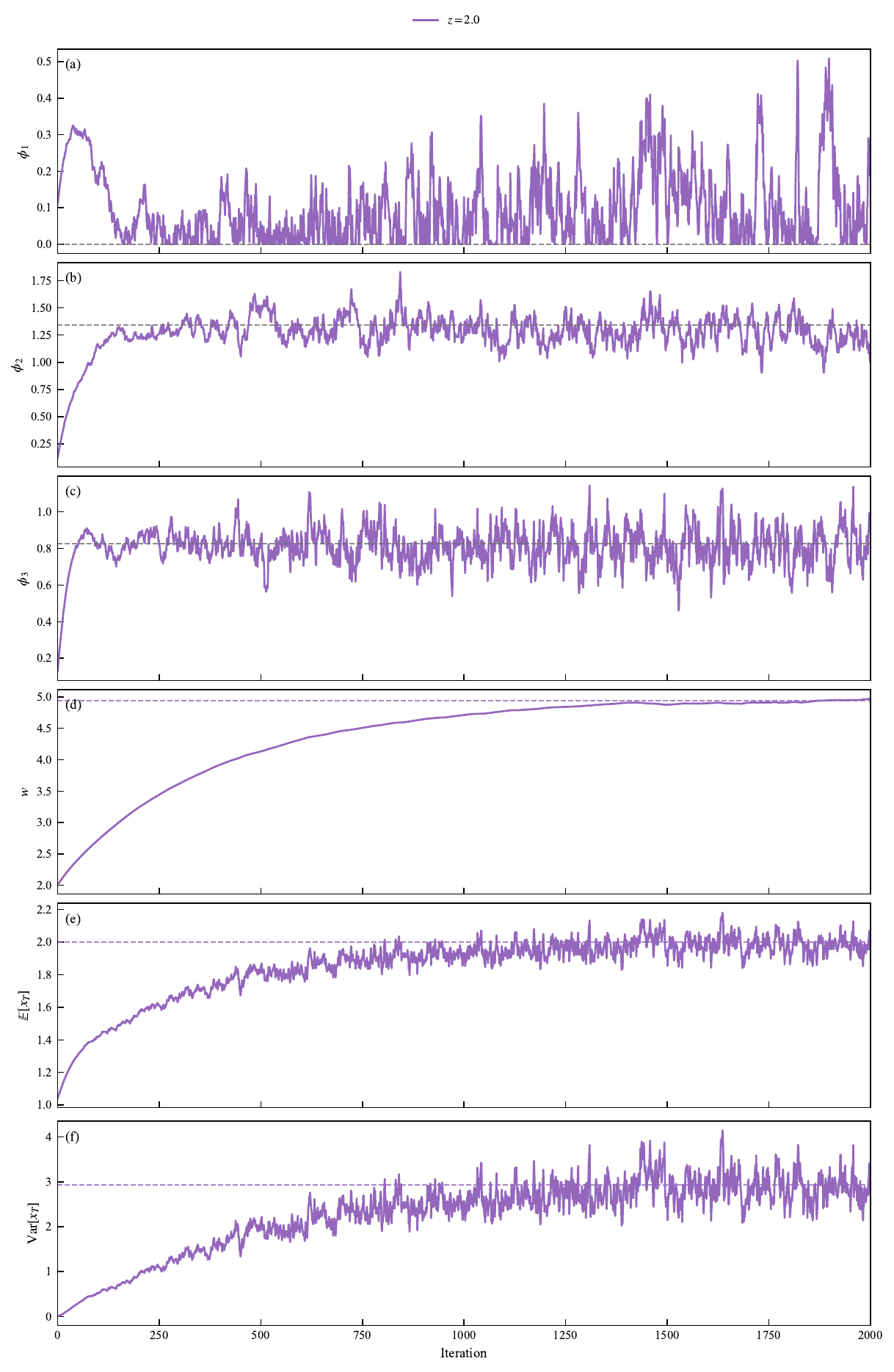}
\caption{\label{fig:results-for-z-extreme}
Iterates of $\phi$, $w$, $\mathbb{E}[x_T]$,
and $\operatorname{Var}(x_T)$ in the main loop of Algorithm~1 for $z=2.0$.
}
\end{figure}

\section{Conclusions}

This paper is the first to treat {\it constrained} continuous-time MV portfolio selection in the realm of reinforcement learning. The main ``trick'' is to take an alternative way of randomization: adding suitably chosen Gaussian noises to the admissible controls, instead of randomizing over the constrained control space. The resulting individually sampled policies may violate the constraint, but their mean is automatically feasible which will be improved until ultimately executed and implemented. This approach was first proposed in \cite{daiDatadrivenMertonsStrategies2025} and has been  applied in different application domains such as generative AI diffusion models \citep{huang2026art} and image generations \citep{huang2026amortized} for {\it unconstrained} problems, but it is potentially more powerful, as foreshadowed by the present paper,  for constrained and other types of RL problems.



\clearpage


\newpage

\section*{Appendix}

\appendix

\section{Viscosity Solutions}
\label{appendix:definition-of-viscosity-solution}
For reader's convenience, we recall the standard definition of viscosity solutions used in the proof of Theorem~\ref{theorem:state-dependent};
see \citet{yongStochasticControlsHamiltonian1999}.

When $\Sigma(t,x) = \Psi_t(x-w)^2$,
the HJB equation \eqref{eq:variance-regularized-HJB} becomes
\begin{equation}
\label{appendix:variance-regularized-HJB}
\left\{
\begin{aligned}
&\frac{\partial}{\partial t} v(t,x) + \inf_{a \in \mathbb{R}^m_+} G\biggl(t,x,a, \frac{\partial v}{\partial x}, \frac{\partial^2 v}{\partial x^2}\biggr) = 0, \\
&v(T,x) = (x-w)^2 - (w-z)^2,
\end{aligned}
\right.
\end{equation}
where
\begin{equation*}
G(t,x,a,p,P) := p a^\intercal \sigma_t \rho_t + \frac{1}{2} P \left[a^\intercal \sigma_t \sigma_t^\intercal a
+ (x-w)^2 \operatorname{tr}(\Psi_t\sigma_t\sigma_t^\intercal) \right].
\end{equation*}

\begin{definition}
A function $v \in C([0,T]\times \mathbb{R})$ is called a \emph{viscosity solution} of \eqref{appendix:variance-regularized-HJB} if
\begin{equation*}
v(T,x) = (x-w)^2 - (w-z)^2,\qquad \forall x \in \mathbb{R},
\end{equation*}
and for all $(t,x) \in [0,T] \times \mathbb{R}$,
\begin{equation*}
\begin{aligned}
&q + \inf_{a \geq 0} G(t,x,a,p,P) \geq 0,
&\qquad &\forall (q,p,P) \in D^{1,2,+}_{t,x}v(t,x),\\
&q + \inf_{a \geq 0} G(t,x,a,p,P) \leq 0,
&\qquad &\forall (q,p,P) \in D^{1,2,-}_{t,x}v(t,x),
\end{aligned}
\end{equation*}
where
\begin{equation*}
\begin{aligned}
D^{1,2,+}_{t,x}v(t_0,x_0) &:= \biggl\{ \biggl( \frac{\partial \varphi}{\partial t}(t_0,x_0), \frac{\partial \varphi}{\partial x}(t_0,x_0), \frac{\partial^2 \varphi}{\partial x^2}(t_0,x_0) \biggm| \varphi \in C^{1,2}([0,T]\times \mathbb{R}) \text{ and } \\
&\qquad\qquad\qquad v - \varphi \text{ attains a local maximum at } (t_0,x_0) \biggr) \biggr\},\\
D^{1,2,-}_{t,x}v(t_0,x_0) &:= \biggl\{ \biggl( \frac{\partial \varphi}{\partial t}(t_0,x_0), \frac{\partial \varphi}{\partial x}(t_0,x_0), \frac{\partial^2 \varphi}{\partial x^2}(t_0,x_0) \biggm| \varphi \in C^{1,2}([0,T]\times \mathbb{R}) \text{ and } \\
&\qquad\qquad\qquad v - \varphi \text{ attains a local minimum at } (t_0,x_0) \biggr) \biggr\}.
\end{aligned}
\end{equation*}
\end{definition}

\begin{remark}
If $v \in C^{1,2}_{t,x}([0,T] \times \mathbb{R})$,
then
\begin{equation*}
\begin{aligned}
D^{1,2,+}_{t,x}v(t_0,x_0) = \biggl\{ \frac{\partial v}{\partial t}(t_0,x_0) \biggr\} \times \biggl\{\frac{\partial v}{\partial x }(t_0,x_0)\biggr\} \times \biggl[\frac{\partial^2 v}{\partial x^2}(t_0,x_0), \infty\biggr),\\
D^{1,2,-}_{t,x}v(t_0,x_0) = \biggl\{ \frac{\partial v}{\partial t}(t_0,x_0) \biggr\} \times \biggl\{\frac{\partial v}{\partial x }(t_0,x_0)\biggr\} \times \biggl(-\infty, \frac{\partial^2 v}{\partial x^2}(t_0,x_0)\biggr].
\end{aligned}
\end{equation*}
In this case,
\begin{equation*}
D^{1,2,+}_{t,x}v(t_0,x_0) \cap D^{1,2,-}_{t,x}v(t_0,x_0) = \biggl( \frac{\partial v}{\partial t}(t_0,x_0), \frac{\partial v}{\partial x }(t_0,x_0), \frac{\partial^2 v}{\partial x^2 }(t_0,x_0)\biggr).
\end{equation*}
\end{remark}

\end{document}